\documentclass[12pt]{article}

\usepackage{amsmath}
\usepackage{amssymb}
\usepackage{graphics}
\usepackage{graphicx}

\newtheorem{theorem}{Theorem}[section]

\newtheorem{proposition}[theorem]{Proposition}
\newtheorem{lemma}[theorem]{Lemma}

\newtheorem{remark}{Remark}

\newenvironment{proof}{\medskip \noindent {\it Proof}.}{\hfill $\Box$\\ \medskip}

\newcommand{\V}{\operatorname{Var}}
\newcommand{\C}{\operatorname{Cov}}
\newcommand{\Vol}{\operatorname{Vol}}
\newcommand{\PV}{\operatorname{PV}}
\newcommand{\A}{\operatorname{A}}
\newcommand{\Lip}{{\operatorname{Lip}}}
\newcommand{\diam}{{\operatorname{diam}}}
\newcommand{\interior}{{\operatorname{int}}}

\newcommand{\proj}{{\operatorname{proj}}}

\def\R{{\mathbb{R}}}

\def\E{{\mathbb{E}}}

\def\N{\mathbb{N}}
\def\P{{\mathbb{P}}}
\def\1{\mbox{{\rm 1 \hskip-1.4ex I}}}

\begin{document}

\title{A Central Limit Theorem for the Poisson-Voronoi Approximation}
\date{}

\author{Matthias Schulte\footnote{University of Osnabrueck, Department of Mathematics, Albrechtstr. 28a, 49076 Osnabrueck, Germany, matthias.schulte[at]uni-osnabrueck.de}}

\maketitle

\begin{abstract}
For a compact convex set $K$ and a Poisson point process $\eta$, the union of all Voronoi cells with a nucleus in $K$ is the Poisson-Voronoi approximation of $K$. Lower and upper bounds for the variance and a central limit theorem for the volume of the Poisson-Voronoi approximation are shown. The proofs make use of so called Wiener-It\^o chaos expansions and the central limit theorem is based on a more abstract central limit theorem for Poisson functionals, which is also derived.
\end{abstract}
\begin{flushleft}
\end{flushleft}


\section{Introduction}

Let $K\subset\R^d, d\geq 2,$ be a compact convex set with interior points and let $\eta$ be a Poisson point process in $\R^d$ with intensity measure $\mu=\lambda\ell_d$ with $\lambda>0$ and the $d$-dimensional Lebesgue measure $\ell_d$. For every point $x\in\eta$ we define the Voronoi cell of $x$ by
$$V_{x}=\left\{z\in\R^d: ||x-z||\leq ||y-z|| \text{ for all } y\in\eta\right\}$$
and call $x$ the nucleus of $V_x$. We have $\interior(V_x)\cap \interior(V_y)=\emptyset$ for $x\neq y\in \eta$ and $\bigcup_{x\in\eta}V_x=\R^d$ such that the collection $\left(V_x\right)_{x\in\eta}$ of random polytopes constitutes a random tessellation of $\R^d$, the so called Poisson-Voronoi tessellation, which is one of the standard models in stochastic geometry, and we refer to \cite{SchneiderWeil2008} and the references therein for further details.

For our set $K$ we define the Poisson-Voronoi approximation $\A(K)$ as
$$\A(K)=\bigcup_{x\in\eta\cap K}V_x,$$
which is a random approximation of $K$. It is possible to interpret the Poisson-Voronoi approximation in the following way. One wants to reconstruct an unknown compact convex set $K$, but the only information available is a kind of oracle which tells for every point of a realization of the Poisson point process if it belongs to $K$. Now one approximates the unknown set $K$ by taking the union of the Voronoi cells with nuclei in $K$.
\begin{figure}
\center
\includegraphics[width=10cm]{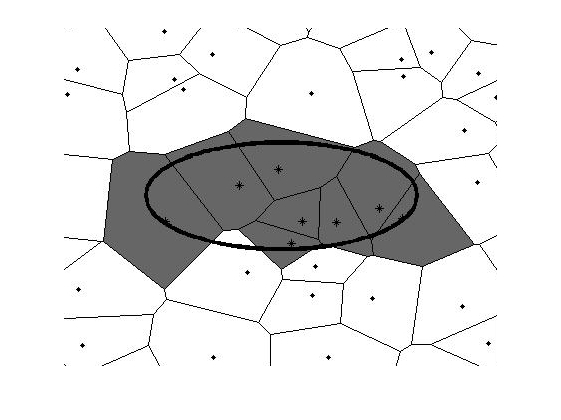}
\caption{Poisson-Voronoi approximation of an ellipse}
\end{figure}

In this paper, we are interested in the volume of the Poisson-Voronoi approximation
$$\PV(K)=\Vol(\A(K)).$$
A short computation yields $\E\PV(K)=\Vol(K)$, which means that $\PV(K)$ is an unbiased estimator for the volume of $K$. Under weaker assumptions than convexity on the approximated set $K$, it was shown in \cite{EinmahlKhmaladze2001,KhmaladzeToronjadze2001,Penrose2007} that $\PV(K)\rightarrow\Vol(K)$ and $\Vol(\A(K)\Delta K)\rightarrow 0$ as $\lambda\rightarrow\infty$, where $\A(K)\Delta K$ stands for the symmetric difference of $A(K)$ and $K$. In \cite{HevelingReitzner2009}, upper bounds for the asymptotic behavior of $\V \PV(K)$ and $\V \Vol(\A(K)\Delta K)$ and large deviation inequalities for $\PV(K)$ and $\Vol(\A(K)\Delta K)$ are derived for the same setting as in this paper. In \cite{Reitzneretal2011}, these results are extended to all non-centered moments and to a more general class of approximated sets, namely sets of finite perimeter. 

The main result of this paper is that $\PV(K)$ behaves asymptotically like a Gaussian random variable if the intensity of the Poisson point process goes to infinity.
\begin{theorem}\label{thm:CLT}
Let $N$ be a standard Gaussian random variable. Then 
$$\frac{PV(K)- \Vol(K)}{\sqrt{\V \PV(K)}}\rightarrow N\ \text{ in distribution as}\ \lambda\rightarrow\infty.$$
\end{theorem}
As it is pointed out in \cite{HevelingReitzner2009}, the Poisson-Voronoi approximation has applications in nonparametric statistics, image analysis and quantization problems. In this context, Theorem \ref{thm:CLT} can be helpful, since it allows to treat $\PV(K)$ as a Gaussian random variable if the intensity of the Poisson point process is sufficiently high.

For the proof of Theorem \ref{thm:CLT} we use so called Wiener-It\^o chaos expansions, which give us a representation
\begin{equation*}
\PV(K)=\Vol(K)+\sum_{n=1}^{\infty}I_n(f_n),
\end{equation*}
where the functions $f_n\in L^2((\R^d)^n),n\in\N,$ are known and $I_n(\cdot)$ denotes the $n$-th multiple Wiener-It\^o integral. The key argument of our proof of Theorem \ref{thm:CLT} is an abstract central limit theorem, which is derived in Section \ref{sec:abstractCLT} and could be helpful for other problems as well. This Theorem \ref{thm:abstractCLT} is based on a general result for the normal approximation of Poisson functionals (see Theorem 3.1 in \cite{Peccatietal2010}), which is used in a similar way as in \cite{ReitznerSchulte2011}.

In order to check the assumptions of this abstract central limit theorem, we have to prove some kind of uniform convergence for $n!||f_n||_n^2$. Combining this property with the identity
$$\V \PV(K) =\sum_{n=1}^\infty n!||f_n||_n^2,$$
we obtain as a byproduct:
\begin{theorem}\label{thm:variance}
Let $r_K$ be the inradius of $K$. Then, there are explicit constants $\underline{C},\overline{C}>0$ depending only on the dimension $d$ such that
\begin{equation}\label{eqn:variance}
\underline{C}\kappa_1V_{d-1}(K)\lambda^{-1-\frac{1}{d}}\leq \V \PV(K) \leq \overline{C}\sum_{i=0}^{d-1}\kappa_{d-i}V_{i}(K)\lambda^{-2+\frac{i}{d}}
\end{equation}
for $\lambda\geq (2/r_K)^d$, where $V_i(K), i=0,\hdots,n$, are the intrinsic volumes of $K$ and $\kappa_j$ stands for the volume of the unit ball in $\R^j$.
\end{theorem}
It is well known that $V_d(K)=\Vol(K)$, $V_{d-1}(K)=\frac{1}{2}S(K)$, where $S(K)$ is the surface area of $K$, and $V_0(K)=1$. Both bounds in (\ref{eqn:variance}) are of order $\lambda^{-1-\frac{1}{d}}$ such that $\V \PV(K)$ also has order $\lambda^{-1-\frac{1}{d}}$.

The upper bound in (\ref{eqn:variance}) is also contained in \cite{HevelingReitzner2009}, where it is proven by a combination of the theory of valuations and the Poincar\'e inequality. The Poincar\'e inequality is also  connected to Wiener-It\^o chaos expansions (for more details we refer to \cite{LastPenrose2011}). The lower bound is new as far as we know.

Although the construction of the Poisson-Voronoi approximation does not depend on the convexity of $K$ and can also be done for more general classes of sets, we formulate our main results only for convex sets, in order to simplify the proofs. In the final Remark \ref{remarkConvexity}, we give two alternative assumptions for the approximated set that allow us to replace the convexity assumption. 

This paper is organized in the following way. In Section \ref{sec:preliminaries}, we introduce Wiener-It\^o chaos expansions and recall the coarea formula we use several times in our proofs. An abstract central limit theorem for Poisson functionals and a helpful proposition to check one of the conditions are derived in Section \ref{sec:abstractCLT}.  In Section \ref{sec:variance}, the Wiener-It\^o chaos expansion for the volume of the Poisson-Voronoi approximation is computed and used to establish Theorem \ref{thm:variance}, before the proof of Theorem \ref{thm:CLT} is concluded in Section \ref{sec:CLTPV}.

\section{Preliminaries}\label{sec:preliminaries}

\paragraph{Wiener-It\^o chaos expansions.} Our main tool in this paper are so called Wiener-It\^o chaos expansions, which are briefly introduced in the following, and we refer to \cite{LastPenrose2011,Peccatietal2010,PeccatiTaqqu2010} for more details. For a Poisson functional $F=F(\eta)$ depending on a Poisson point process $\eta$ over a Borel space $(X,\mathcal{X},\mu)$ with a $\sigma$-finite non-atomic intensity measure $\mu$ (for $\PV(K)$: $X=\R^d$, ${\cal X}$ is the standard Borel $\sigma$-field in $\R^d$ and $\mu=\lambda\ell_d$) one defines the difference operator as
$$D_xF=F(\eta+\delta_x)-F(\eta)$$
for $x\in X$, where $\delta_x$ is the Dirac measure concentrated at the point $x$. The difference operator has the geometric interpretation that it measures the effect of adding the point $x$ to the Poisson point process $\eta$. Therefore, it is sometimes called add-one-cost operator. The iterated difference operator is given recursively by
\begin{equation}\label{eqn:definitionD}
D_{x_1,\hdots,x_n}F=D_{x_1}D_{x_2,\hdots,x_n}F
\end{equation}
and is symmetric under all permutations of $x_1,\hdots,x_n$. The definition in (\ref{eqn:definitionD}) is equivalent to
\begin{equation}\label{eqn:formulaD}
D_{x_1,\hdots,x_n}F=\sum_{I\subset \{1,\hdots,n\}}(-1)^{n+|I|}F(\eta+\sum_{i\in I}\delta_{x_i}).
\end{equation}
Denoting the $n$-th multiple Wiener-It\^o integral with respect to the compensated Poisson point process $\eta-\mu$ by $I_n(\cdot)$ and defining functions $f_n: X^n\rightarrow\overline{\R}:=\R\cup\left\{\pm\infty\right\}$ as
\begin{equation}\label{eqn:definitionfn}
f_n(x_1,\hdots,x_n)=\frac{1}{n!}\E D_{x_1,\hdots,x_n}F,
\end{equation}
we have the following representation for square integrable Poisson functionals $F$ (see Theorem 1.1 and Theorem 1.3 in \cite{LastPenrose2011}):

\begin{theorem}\label{thm:chaos}
Let $F\in L^2(\P)$. Then $f_n\in L^2(X^n)$ for $n\in\N$ and
\begin{equation}\label{eqn:chaosexpansion}
F=\E F+\sum_{n=1}^{\infty} I_n(f_n).
\end{equation}
Moreover,
\begin{equation}\label{eqn:chaosvariance}
\V F=\sum_{n=1}^{\infty} n!||f_n||_n^2,
\end{equation}
where $||\cdot||_n$ stands for the usual norm in $L^2(X^n)$.
\end{theorem}
We call the identity (\ref{eqn:chaosexpansion}) the Wiener-It\^o chaos expansion of $F$ and the functions $f_n$ the kernels of the Wiener-It\^o chaos expansion of $F$. Note that (\ref{eqn:chaosexpansion}), which is sometimes denoted as Fock space representation, is an orthogonal series since 
$$\E I_n(f)I_m(g)=\begin{cases}n! \int_{X^n}fg\ d\mu^{n}, &n=m\\ 0, &n\neq m \end{cases}.$$
Analogously to (\ref{eqn:variance}) the covariance of two Poisson functionals $F,G\in L^2(\P)$ with $F=\E F+\sum_{n=1}^\infty I_n(f_n)$ and $G=\E G+\sum_{n=1}^\infty I_n(g_n)$ is given by
\begin{equation}\label{eqn:covariance}
\C(F,G)=\sum_{n=1}^\infty n!\langle f_n,g_n\rangle_n,
\end{equation}
where $\langle\cdot,\cdot\rangle_n$ is the usual inner product in $L^2(X^n)$.

In the next section, we need the inverse Ornstein-Uhlenbeck generator $L^{-1}$ which is for centred random variables with a chaos expansion (\ref{eqn:chaosexpansion}) given by
\begin{equation}\label{eqn:OrnsteinUhlenbeck}
L^{-1}F=-\sum_{n=1}^\infty \frac{1}{n} I_n(f_n).
\end{equation}
In this context, it is also possible to define the difference operator as $$D_xF=\sum_{n=1}^\infty n I_{n-1}(f_n(x,\cdot))$$ if $\sum_{n=1}^\infty nn!\|f_n\|_n^2<\infty$.

\paragraph{Coarea formula.} Our main tool for the computation of integrals where the kernels of the chaos expansion of $F$ arise, e.g. $\|f_n\|_n^2$, is the so called coarea formula. By ${\cal H}^m$, we denote the $m$-dimensional Hausdorff measure.  If $f: \R^m\rightarrow\R^n$ is differentiable in $x\in \R^m$, we define the Jacobian $J f(x)$ by
$$J f(x)=\sqrt{\det\left(f'(x)\ f'(x)^T\right)},$$
where $f'$ stands for the Jacobi matrix of $f$. Note that a Lipschitz function is almost everywhere differentiable such that its Jacobian is almost everywhere defined. Using this notation, we have (see Corollary 5.2.6 in \cite{KrantzParks2008}, for example):
\begin{theorem}\label{thm:coarea}
If $f:\R^m\rightarrow\R^n$ is a Lipschitz function and $m\geq n$, then
\begin{equation*}
\int_B g(x) Jf(x) \ell_m(dx)=\int_{\R^n}\int_{B\cap f^{-1}(y)}g(z) {\cal H}^{m-n}(dz)\ell_n(dy)
\end{equation*}
holds for each Lebesgue measurable $B\subset \R^m$ and each nonnegative $\ell_m$-measurable function $g: B\rightarrow\R$.
\end{theorem}
For $n=1$, we have $Jf(x)=||\nabla f(x)||$. Note that $||\cdot||$ stands for the usual Euclidean norm, whereas $||\cdot||_n$ is the norm in $L^2(X^n)$ or $L^2((\R^d)^n)$.

\section{An abstract central limit theorem}\label{sec:abstractCLT}

In this section, we prove an abstract central limit theorem for a more general setting, which is used to show Theorem \ref{thm:CLT}. As in the previous section, we assume that $\eta$ is a Poisson point process over a Borel space $(X, {\cal{X}}, \mu)$ with a $\sigma$-finite non-atomic intensity measure $\mu$ and $F\in L^2(\P)$ is a Poisson functional with a Wiener-It\^o chaos expansion
$$F=\E F+\sum_{n=1}^{\infty}I_n(f_n).$$
We are interested in the Wasserstein distance between our Poisson functional and a Gaussian random variable, which is very helpful for establishing central limit theorems since convergence in Wasserstein distance implies convergence in distribution. The Wasserstein distance $d_W$ of two random variables $Y$ and $Z$ is given by
\begin{equation}\label{eqn:defWasserstein}
d_W(Y,Z)=\sup\limits_{h\in\Lip(1)}|\E h(Y)-\E h(Z)|,
\end{equation}
where $\Lip(1)$ is the set of all functions $h:\R\rightarrow\R$ with a Lipschitz constant less or equal than one. In this setting, we can state the following:

\begin{theorem}\label{thm:abstractCLT}
Let $F\in L^2(\P)$ and let $N$ be a standard Gaussian random variable. 
\begin{itemize}
 \item [a)] For every $k\in\mathbb{N}$ one has
\begin{equation}\label{eqn:inequalitydW}
\!\!\!\!\!\!\!\!\!\!\!\!\!\!\!\!\!\!\!\! d_W\left(\frac{F-\E F}{\sqrt{\V F}},N\right)\leq 2\frac{\sqrt{\sum_{n=k+1}^{\infty} n!||f_n||_n^2}}{\sqrt{\V F}}+k\sum_{1\leq i,j\leq k}\frac{\sqrt{R_{ij}}}{\V F}+k^{\frac{7}{2}}\sum_{i=1}^k\frac{\sqrt{\tilde{R_i}}}{\V F}
\end{equation}
with 
\begin{eqnarray*}
R_{ij}&=& \E\langle I_{i-1}(f_i(z,\cdot)),I_{j-1}(f_j(z,\cdot))\rangle^2-\left(\E\langle I_{i-1}(f_i(z,\cdot)),I_{j-1}(f_j(z,\cdot))\rangle\right)^2\\
\tilde{R}_i&=&\E\langle I_{i-1}(f_i(z,\cdot))^2,I_{i-1}(f_i(z,\cdot))^2\rangle
\end{eqnarray*}
for $i,j=1,\hdots,k$. Here $\langle\cdot,\cdot\rangle$ stands for the usual inner product in $L^2(X)$.
\item [b)] Let $F$ depend on a parameter $\lambda>0$. If there are constants $c_n\in\R$ such that 
\begin{equation}\label{eqn:conditionkernels}
\frac{n! ||f_n||_n^2}{\V F}\leq c_n\ \text{ for all }\ \lambda>\lambda_0\in\R\ \text{ and } \sum c_n<\infty
\end{equation}
and
\begin{equation}\label{eqn:conditionRs}
\frac{\sqrt{R_{ij}}}{\V F},\frac{\sqrt{\tilde{R}_i}}{\V F}\rightarrow 0\ \text{ as } \lambda\rightarrow\infty
\end{equation} 
for all $i,j\in \mathbb{N}$, then
$$d_W\left(\frac{F-\E F}{\sqrt{\V F}},N\right)\rightarrow 0\ \text{as}\ \lambda\rightarrow\infty.$$
\end{itemize}
\end{theorem}
Since $F$ is standardized in Theorem \ref{thm:abstractCLT}, we can assume without loss of generality that $\E F=0$. The idea of the proof of (\ref{eqn:inequalitydW}) is to define truncated Poisson functionals $F_k$, $k\in\mathbb{N}$, by
$$F_k=\sum_{n=1}^k I_n(f_n)$$
and to use the triangle inequality
\begin{equation}\label{eqn:bounddW}
d_W\left(\frac{F}{\sqrt{\V F}}, N\right)\leq d_W\left(\frac{F}{\sqrt{\V F}},\frac{F_k}{\sqrt{\V F}}\right)+d_W\left(\frac{F_k}{\sqrt{\V F}},N\right).
\end{equation}
Now we compute upper bounds for both expressions on the right hand side in (\ref{eqn:bounddW}).

\begin{lemma}\label{lem:dW1}
It holds
\begin{equation}\label{eqn:dW1}
d_W\left(\frac{F}{\sqrt{\V F}},\frac{F_k}{\sqrt{\V F}}\right)\leq \sqrt{\frac{\sum_{n=k+1}^{\infty}n!||f_n||_n^2}{\V F}}
\end{equation}
for every $k\in\mathbb{N}$.
\end{lemma}
\begin{proof}
By the definition of the Wasserstein distance in (\ref{eqn:defWasserstein}) and the Cauchy-Schwarz inequality, we obtain
\begin{eqnarray*}
d_W\left(\frac{F}{\sqrt{\V F}},\frac{F_k}{\sqrt{\V F}}\right)  &=& \sup\limits_{h\in \Lip(1)}\left|\E h\left(\frac{F}{\sqrt{\V F}}\right)-\E h\left(\frac{F_k}{\sqrt{\V F}}\right)\right|\\ &\leq& \E \left|\frac{F}{\sqrt{\V F}}-\frac{F_k}{\sqrt{\V F}}\right|= \frac{\E|F-F_k|}{\sqrt{\V F}}\\ &\leq& \frac{\sqrt{\E(F-F_k)^2}}{\sqrt{\V F}}
\leq \frac{\sqrt{\E\left(\sum_{n=k+1}^{\infty}I_n(f_n)\right)^2}}{\sqrt{\V F}}
\end{eqnarray*}
and (\ref{eqn:dW1}) is a direct consequence of (\ref{eqn:chaosvariance}).

\end{proof}
For the second expression in (\ref{eqn:bounddW}) we need the following inequality:

\begin{lemma}\label{lem:dW2}
We have
\begin{equation}\label{eqn:dW2}
\!\!\!\!\!\!\!d_W\left(\frac{F_k}{\sqrt{\V F}},N\right)\leq \frac{\sum_{n=k+1}^{\infty} n!||f_n||_n^2}{\V F}+k\sum_{1\leq i,j\leq k}\frac{\sqrt{R_{ij}}}{\V F}+k^{\frac{7}{2}}\sum_{i=1}^k\frac{\sqrt{\tilde{R_i}}}{\V F}
\end{equation}
for every $k\in\mathbb{N}$.
\end{lemma}

\begin{proof}
Theorem 3.1 in \cite{Peccatietal2010} tells us that
\begin{eqnarray*}
d_W\left(\frac{F_k}{\sqrt{\V F}},N\right) &\leq& \E\left|1-\frac{1}{\V F}\int_X D_z F_k \left(-D_zL^{-1}F_k\right)\mu(dz)\right|\\ &&+\frac{1}{(\V F)^{\frac{3}{2}}}\int_X \E\left[\left(D_z F_k\right)^2|D_zL^{-1}F_k|\right]\mu(dz),
\end{eqnarray*}
where $L^{-1}$ is the inverse of the Ornstein-Uhlenbeck generator as given in (\ref{eqn:OrnsteinUhlenbeck}). By the fact that $1=\sum_{n=1}^k n! ||f_n||_n^2/\V F+\sum_{n=k+1}^\infty n! ||f_n||_n^2/\V F$ and the triangle inequality, we obtain
\begin{eqnarray*}
d_W\left(\frac{F_k-\E F}{\sqrt{\V F}},N\right) &\leq& \frac{\sum_{n=k+1}^{\infty} n!||f_n||_n^2}{\V F}\\ &&+\frac{1}{\V F}\E\left|\sum_{n=1}^k n!||f_n||_n^2-\int_X D_zF_k \left(-D_zL^{-1}F_k\right)\mu(dz)\right|\\ &&+\frac{1}{(\V F)^{\frac{3}{2}}}\int_X \E\left[\left(D_zF_k\right)^2|D_zL^{-1}F_k|\right]\mu(dz).
\end{eqnarray*}
In the proof of Theorem 4.1 in \cite{ReitznerSchulte2011}, it is shown that the last two expressions are bounded by 
$$ k\sum_{1\leq i,j\leq k}\frac{\sqrt{R_{ij}}}{\V F}+k^{\frac{7}{2}}\sum_{i=1}^k\frac{\sqrt{\tilde{R_i}}}{\V F},$$
which leads to (\ref{eqn:dW2}).
\end{proof}
\paragraph{Proof of Theorem \ref{thm:abstractCLT}.} The inequality (\ref{eqn:inequalitydW}) in a) is a direct consequence of (\ref{eqn:bounddW}), Lemma \ref{lem:dW1} and Lemma \ref{lem:dW2}. If the conditions of b) are satisfied, for every $\varepsilon>0$ we can find $k_0\in\mathbb{N}$ such that
$$2\frac{\sqrt{\sum_{n=k_0+1}^{\infty}n! ||f_n||_n^2}}{\sqrt{\V F}}<\frac{\varepsilon}{2}$$
for all $\lambda>\lambda_0$. Because of (\ref{eqn:conditionRs}) it exists a constant $\overline{\lambda}>0$ such that
$$k_0\sum_{1\leq i,j\leq k_0}\frac{\sqrt{R_{ij}}}{\V F}+k_0^{\frac{7}{2}}\sum_{i=1}^{k_0}\frac{\sqrt{\tilde{R_i}}}{\V F}<\frac{\varepsilon}{2}$$ 
for all $\lambda>\overline{\lambda}$. Combining these inequalities with (\ref{eqn:inequalitydW}), we obtain
$$d_W\left(\frac{F-\E F}{\sqrt{\V F}},N\right)<\varepsilon$$
for all $\lambda>\max\left\{\lambda_0,\overline{\lambda}\right\}$.\hfill $\Box$ \\ 

\begin{remark}\rm
Our abstract central limit Theorem \ref{thm:abstractCLT} and Theorem \ref{thm:CLT} have the drawback that they do not give a rate of convergence. This problem is caused by the truncation step. The second expression in (\ref{eqn:bounddW}) vanishes as $\lambda\rightarrow\infty$. But the first summand does not necessarily converge for a fixed $k$ as $\lambda\rightarrow\infty$. By taking $k\rightarrow\infty$ in the previous proof, we obtain convergence as $\lambda\rightarrow\infty$, but cannot give a rate of convergence. An alternative approach would be to apply the underlying general central limit theorem (see Theorem 3.1 in \cite{Peccatietal2010}) directly to the infinite Wiener-It\^o chaos expansion, which gives a sum of an infinite number of expected values of products of multiple Wiener-It\^o integrals as an upper bound. These summands are also the $R_{ij}$ and $\tilde{R}_i$ in our Theorem \ref{thm:abstractCLT} and it is possible to show an upper bound and a rate of convergence for each of them as one will see in Section \ref{sec:CLTPV}. But it seems very hard to prove that the series over these bounds converges.
\end{remark}

In order to neatly formulate a very helpful criterion for the condition (\ref{eqn:conditionRs}) in Theorem \ref{thm:abstractCLT}, we need the following notation:
Let $\Pi_{i,j}$ be the set of all partitions $\pi$ of the variables $$x_1^{(1)},\hdots,x_{i}^{(1)},x_1^{(2)},\hdots,x_{i}^{(2)},x_1^{(3)},\hdots,x_{j}^{(3)},x_1^{(4)},\hdots,x_{j}^{(4)}$$
such that
\begin{itemize}
 \item all variables with the same upper index are in different elements of $\pi$
\item every element of $\pi$ has at least two variables as elements.
\end{itemize}
By $\overline{\Pi}_{i,j}$, we denote the set of all partitions $\pi$ in $\Pi_{i,j}$ such that for any decomposition of $\left\{1,2,3,4\right\}$ in two disjoint nonempty sets $M_1$ and $M_2$ there exist $l_1\in M_1$ and $l_2\in M_2$ such that two variables $x^{(l_1)}_{k_1}$ and $x^{(l_2)}_{k_2}$ are in the same element of $\pi$.

We say that a partition $\pi$ connects two variables if they are in the same element of $\pi$, and $|\pi|$ stands for the number of elements of a partition $\pi\in\Pi_{i,j}$. By $(f_i*f_i*f_j*f_j)_{\pi}$, we denote the function from $X^{|\pi|}$ to $\R$ we obtain if we replace all variables that are in the same element of $\pi$ by a new variable. Using this notation, we can give the following upper bounds for $R_{ij}$ and $\tilde{R}_i$:

\begin{proposition}\label{prop:product}
If $\int_{X^{|\pi|}}|(f_i*f_i*f_i*f_i)_{\pi}|d\mu^{|\pi|}<\infty$ for all $\pi\in\overline{\Pi}_{i,i}$ and\\ $\int_{X^{|\pi|}}|(f_j*f_j*f_j*f_j)_{\pi}|d\mu^{|\pi|}<\infty$ for all $\pi\in\overline{\Pi}_{j,j}$, it holds
\begin{equation}\label{eqn:boundRij}
R_{ij}\leq\sum_{\pi\in\overline{\Pi}_{i,j}} \int_{X^{|\pi|}}|(f_i*f_i*f_j*f_j)_{\pi}|d\mu^{|\pi|}.
\end{equation}
Moreover, one has
\begin{equation}\label{eqn:boundRi}
\tilde{R}_{i}\leq\sum_{\pi\in\overline{\Pi}_{i,i}} \int_{X^{|\pi|}}|(f_i*f_i*f_i*f_i)_{\pi}|d\mu^{|\pi|}.
\end{equation}
\end{proposition}
\begin{proof}
The fact that the integrals over $(f_i*f_i*f_i*f_i)_{\pi}$ and $(f_j*f_j*f_j*f_j)_{\pi}$ are finite ensures that we can apply the product formula for multiple Wiener-It\^o integrals (see Theorem 3.1 in \cite{Surgailis1984} or Proposition 6.5.1 in \cite{PeccatiTaqqu2010}) to $I_{i-1}(f_i(s,\cdot))I_{i-1}(f_i(t,\cdot))$ and $I_{j-1}(f_j(s,\cdot))I_{j-1}(f_j(t,\cdot))$ for $\mu$-almost all $(s,t)\in X^2$. This formula gives us the kernels of the chaos expansions of $I_{i-1}(f_i(s,\cdot))I_{i-1}(f_i(t,\cdot))$ and $I_{j-1}(f_j(s,\cdot))I_{j-1}(f_j(t,\cdot))$. Combining this with the covariance formula (\ref{eqn:covariance}), we know that
\begin{eqnarray*}
&&\E\int_X\int_X I_{i-1}(f_i(s,\cdot)) I_{i-1}(f_i(t,\cdot))I_{j-1}(f_j(s,\cdot)) I_{j-1}(f_j(t,\cdot))\mu(ds)\mu(dt)\\
&=&\sum_{\pi\in\Pi_{i-1,j-1}} \int_X\int_X\int_{X^{|\pi|}}(f_i(s,\cdot)*f_i(t,\cdot)*f_j(s,\cdot)*f_j(t,\cdot))_{\pi} d\mu^{|\pi|}\mu(ds)\mu(dt).
\end{eqnarray*}
For $i\neq j$ there is no partition connecting either variables of the first and the third function or variables of the second and the fourth function. Hence, we have only partitions from $\overline{\Pi}_{i,j}$ if we add $s$ and $t$ to the partitions. For $i=j$ we have $(i-1)!(i-1)!$ partitions where the variables of the first and third and of the second and fourth function are pairwise connected. The sum over this partitions is $(i-1)!(i-1)!||f_i||_i^4$. But exactly this term is subtracted for $i=j$. If we add $s$ and $t$ to the remaining partitions, we also obtain partitions from $\overline{\Pi}_{i,i}$.
For $\tilde{R}_i$ we can also apply the product formula and it holds
\begin{eqnarray*}
\tilde{R}_i &=& \E\int_X I_{i-1}(f_i(z,\cdot))^4\mu(dz)\\
&=&\sum_{\pi\in\Pi_{i-1,i-1}} \int_X \int_{X^{|\pi|}} (f_i(z,\cdot)*f_i(z,\cdot)*f_i(z,\cdot)*f_i(z,\cdot))_{\pi}d\mu^{|\pi|}\mu(dz)\\
&\leq & \sum_{\pi\in\overline{\Pi}_{i,j}}\int_X\int_{X^{|\pi|}}|(f_i*f_i*f_i*f_i)_{\pi}|d\mu^{|\pi|}.
\end{eqnarray*}
\end{proof}

\section{Proof of Theorem \ref{thm:variance}}\label{sec:variance}

Because of Theorem 1 in \cite{HevelingReitzner2009}, we know that $\PV(K)\in L^2(\P)$ and Theorem \ref{thm:chaos} implies the existence of a Wiener-It\^o chaos expansion. In the following, we compute the kernels of this decomposition and use (\ref{eqn:chaosvariance}) to prove our bounds for the variance of $\PV(K)$ in Theorem \ref{thm:variance}.

From now on, we denote by $\rho(\cdot,\cdot)$ the usual Euclidean distance, which is given by $\rho(x,y)=||x-y||$ for two points $x,y\in\R^d$ and $\rho(x, A)=\inf\limits_{y\in A}\rho(x,y)=\inf\limits_{y\in A}||x-y||$ for $x\in \R^d$ and $A\subset \R^d$. Moreover, $B^d(z,\delta)$ stands for a ball with center $z$ and radius $\delta$ in $\R^d$. Using this notation, we have the following formula for the kernels of the Wiener-It\^o chaos expansion of $\PV(K)$:

\begin{lemma}\label{lem:kernels}
Let $x_1,\hdots,x_n\in \R^d$. For $y\in\R^d$ we define $\overline{x}(y):=\arg\max_{x=x_1,\hdots,x_n}\rho(y,x)$ and $z(y,\eta):=\arg\min_{z\in\eta}\rho(y,z)$. Then
\begin{eqnarray}\label{eqn:kernels}
f_n(x_1,\hdots,x_n)\! &=& \!\frac{(-1)^n}{n!}\int_{\R^d}\!\!\!\1(\overline{x}(y)\!\notin\!K)\P(z(y,\eta)\notin K^C\!\!\cup B^d(y,||y-\overline{x}(y)||))dy\\ &&-\frac{(-1)^n}{n!}\int_{\R^d}\!\!\!\1(\overline{x}(y)\!\in\! K)\P(z(y,\eta)\notin K\cup B^d(y,||y-\overline{x}(y)||))dy.\notag
\end{eqnarray}
\end{lemma}
\begin{proof}
Since $z(y,\eta)$ is the nucleus of the Voronoi cell $y$ belongs to, it is easily seen that
$$\PV(K)=\Vol(\{y\in\R^d: z(y,\eta)\in K\})=\int_{\R^d} \1(z(y,\eta)\in K)dy.$$
Combining this with (\ref{eqn:formulaD}), we obtain
\begin{eqnarray*}
D_{x_1,\hdots,x_n}\PV(K) &=& \sum_{I\subset\{1,\hdots,n\}} (-1)^{n+|I|} \PV(K)(\eta+\sum_{i\in I}\delta_{x_i})\\
&=& \int_{\R^d} \sum_{I\subset\{1,\hdots,n\}} (-1)^{n+|I|} \1(z(y,\eta\cup\{x_i:i\in I\})\in K) dy.
\end{eqnarray*}
Now we consider the sum of the indicator functions on the right hand side for a fixed $y\in K$. Let $i_{max}$ be the index of the $x_i$ that maximizes $\overline{x}(y)$. For $I\subset \{1,\hdots,n\}\setminus\{i_{max}\}$ with $I\neq\emptyset$, it holds $z(y,\eta\cup\{x_i:i\in I\})=z(y,\eta\cup\{x_i:i\in I\cup\{i_{max}\}\})$ and the summands for $I$ and $I\cup\{i_{max}\}$ on the right hand side cancel out because of the different signs. Hence, we obtain
$$D_{x_1,\hdots,x_n}\PV(K)=\int_{\R^d}(-1)^{n}(\1(z(y,\eta)\in K)-\1(z(y,\eta\cup\{\overline{x}(y)\})\in K))dy.$$
Now it is easy to see that
\begin{eqnarray*}
&&\1(z(y,\eta)\in K)-\1(z(y,\eta\cup\{\overline{x}(y)\})\in K)\\ &&=\begin{cases} 1, & \rho(y,\overline{x}(y))\leq \rho(y,z(y,\eta)),\ z(y,\eta)\in K,\ \overline{x}(y)\notin K\\ -1, & \rho(y,\overline{x}(y))\leq \rho(y,z(y,\eta)),\ z(y,\eta)\notin K,\ \overline{x}(y)\in K \end{cases}.
\end{eqnarray*}
Combining this with the definition of the kernels in (\ref{eqn:definitionfn}), we obtain
\begin{eqnarray*}
f_n(x_1,\hdots,x_n) &=& \frac{(-1)^n}{n!}\int_{\R^d}\1(\overline{x}(y)\notin K)\P(z(y,\eta)\notin K^C\cup B^d(y,||y-\overline{x}(y)||))dy\\ &&-\frac{(-1)^n}{n!}\int_{\R^d}\1(\overline{x}(y)\in K)\P(z(y,\eta)\notin K\cup B^d(y,||y-\overline{x}(y)||))dy.
\end{eqnarray*}
\end{proof}
\begin{remark}\rm
For $f_1$ we have the representation
$$f_1(x)=\begin{cases}\E\Vol(\{z\in\R^d:\rho(z,x)\leq \rho(z,\eta\cap K^C)\leq \rho(z,\eta\cap K)\}), &x\in K\\ -\E\Vol(\{z\in\R^d:\rho(z,x)\leq \rho(z,\eta\cap K)\leq \rho(z,\eta\cap K^C)\}), &x\in K^C\end{cases},$$
which means that $|f_1(x)|$ is exactly the volume of the points that change between $\A(K)$ and $\A(K)^C$ if the point $x$ is added to the Poisson point process.
\end{remark}
Our next goal is to compute upper bounds for $||f_n||_n^2$ such that we obtain by (\ref{eqn:chaosvariance}) an upper bound for the variance of $\PV(K)$ and can check condition (\ref{eqn:conditionkernels}) in Theorem \ref{thm:abstractCLT}.
In formula (\ref{eqn:kernels}), the distance between a point $y\in \R^d$ and $\overline{x}(y)$ plays an important r\^ole. In order to handle this quantity in the following, we define functions $h_n: (\R^d)^n\rightarrow\R\times\R^d$ by
$$h_n(x_1,\hdots,x_n)=(\min_{y\in \R}\max_{i=1,\hdots,n}\rho(y,x_i),\arg\min_{y\in \R}\max_{i=1,\hdots,n}\rho(y,x_i)).$$
From a geometrical point of view, $h_n$ gives the radius and the center of the smallest ball that contains all points $x_1,\hdots,x_n$.

The function $h_n$ allows us to give the following upper bound for $f_n$:

\begin{lemma}\label{lem:Boundfn1}
Let $x_1,\hdots,x_n\in\R^d$ and let $r=h_n^{(1)}(x_1,\hdots,x_n)=\min\limits_{y\in \R}\max\limits_{i=1,\hdots,n}\rho(y,x_i)$. Then
\begin{equation}\label{eqn:Boundfn1}
|f_n(x_1,\hdots,x_n)|\leq \frac{1}{(n-1)!\lambda}\exp(-\lambda \kappa_d r^d).
\end{equation}
\end{lemma}
\begin{proof}
As a consequence of Lemma \ref{lem:kernels}, one has
$$|f_n(x_1,\hdots,x_n)| \leq \frac{1}{n!}\int_{\R^d}\P(z(y,\eta)\notin B^d(y,||y-\overline{x}(y)||))dy.$$
By the definition of $r$, we know that the sets $\R^d\setminus B^d(x_i,r), i=1,\hdots,n,$ cover $\R^d$. Combining this with the previous inequality and using polar coordinates, we have
\begin{eqnarray*}
|f_n(x_1,\hdots,x_n)| &\leq& \frac{1}{n!}\sum_{i=1}^n \int_{\R^d\setminus B^d(x_i,r)}\P(z(y,\eta)\notin B^d(y,||x_i-y||))dy\\ 
&=& \frac{1}{n!}\sum_{i=1}^n \int_{\R^d\setminus B^d(x_i,r)}\exp(-\lambda \kappa_d\|x_i-y\|^d)dy  \\
&=& \frac{1}{n!}\sum_{i=1}^n \kappa_d d\int_r^\infty \exp(-\lambda\kappa_d r^d) r^{d-1} dr\\ &=& \frac{1}{(n-1)!\lambda}\exp(-\lambda \kappa_d r^d).
\end{eqnarray*}
\end{proof}
By definition, $f_n(x_1,\hdots,x_n)$ measures the effect on $\PV(K)$ of inserting points. Lemma \ref{lem:Boundfn1} reflects the fact that this effect is small if the distances between the points are large. Similar one expects that $f_n(x_1,\hdots,x_n)$ is small if all points are close together but are far away from the boundary of $K$. This effect is described by the following Lemma:

\begin{lemma}\label{lem:Boundfn2}
Let $x_1,\hdots,x_n\in\R^d$ and $(r,\overline{y})=h_n(x_1,\hdots,x_n)$. If $\delta=\rho(\overline{y},\partial K)>8r$, then 
\begin{equation}\label{eqn:Boundfn2}
|f_n(x_1,\hdots,x_n)|\leq \frac{2}{n!\lambda}\exp(-\lambda\kappa_d\delta^d/8^d ).
\end{equation}
\end{lemma}
\begin{proof}
Since $\rho(\overline{y},\partial K)>8r$, all $x_1,\hdots,x_n$ are either in $K$ or $K^C$. Let $\tilde{x}=\frac{1}{2}(\overline{y}+\proj_{\partial K}(\overline{y}))$, where $\proj_{\partial K}(\overline{y})$ stands for the metric projection of $\overline{y}$ on the boundary of $K$. If $\overline{y}\in K$, it can happen that the metric projection on $\partial K$ is not unique. In this case, it does not matter which of the points is taken. Then, we have
$$\frac{\delta}{4}\leq \rho(\tilde{x},y)\leq \frac{3}{4}\delta\leq \rho(y,\partial K)\ \text{ for all }y\in B^d(x_1,\delta/8)$$
and, by (\ref{eqn:kernels}), it follows
\begin{eqnarray*}
|f_n(x_1,\hdots,x_n)| & \leq &  \frac{1}{n!}\int_{\R^d\setminus B^d(x_1,\delta/8)}\P(z(y,\eta)\notin B^d(y,||y-x_1||))dy\\ &&+\frac{1}{n!}\int_{B^d(x_1,\delta/8)}\P(z(y,\eta)\notin B^d(y,\rho(y,\partial K)))dy\\
& \leq &  \frac{1}{n!}\int_{\R^d\setminus B^d(x_1,\delta/8)}\P(z(y,\eta)\notin B^d(y,||y-x_1||))dy\\ &&+\frac{1}{n!}\int_{\R\setminus B^d(\tilde{x},\delta/8)}\P(z(y,\eta)\notin B^d(y,||\tilde{x}-y||))dy.
\end{eqnarray*}
A straightforward computation as in Lemma \ref{lem:Boundfn1} yields (\ref{eqn:Boundfn2}).
\end{proof}\\
Combining Lemma \ref{lem:Boundfn1} and Lemma \ref{lem:Boundfn2} leads to the bound
$$|f_n(x_1,\hdots,x_n)|\leq \overline{f}_n(h_n(x_1,\hdots,x_n)),$$
where $\overline{f}_n: \R\times\R^d\rightarrow\R$ is given by
\begin{eqnarray*}
\overline{f}_n(r,y)=\begin{cases} \frac{1}{(n-1)!\lambda}\exp(-\lambda\kappa_dr^d), & \rho(y,\partial K)\leq8r\\ \frac{2}{n!\lambda}\exp(-\lambda\kappa_dr^d), & \rho(y,\partial K)>8r\end{cases}.
\end{eqnarray*}
By the coarea formula Theorem \ref{thm:coarea}, we obtain for $n\geq 2$
\begin{eqnarray}\label{eqn:transformation1}
||f_n||_n^2 & \leq & \lambda^n\int_{(\R^d)^n} \overline{f}_n(h_n(x_1,\hdots,x_n))^2 dx_1\hdots dx_n\\
&=& \lambda^n\int_0^\infty\int_{\R^d}\int_{h_n^{-1}(r,y)} \overline{f}_n(r,y)^2 Jh_n(x_1,\hdots,x_n)^{-1} {\cal H}^{nd-d-1}(d(x_1,\hdots,x_n))dy dr\notag\\
&=& \lambda^n\int_0^\infty\int_{\R^d}\overline{f}_n(r,y)^2 \int_{h_n^{-1}(r,y)} Jh_n(x_1,\hdots,x_n)^{-1} {\cal H}^{nd-d-1}(d(x_1,\hdots,x_n))dy dr.\notag
\end{eqnarray}
It is easy to see that $h_n(a x_1+v,\hdots,a x_n+v)=a h_n(x_1,\hdots,x_n)+(0,v)$ for all $a>0$ and $v\in\R^d$ and a short computation shows $h_n'(ax_1+v,\hdots,ax_n+v)=h_n'(x_1,\hdots,x_n)$, which implies $$Jh_n(a x_1+v,\hdots,a x_n+v)=Jh_n(x_1,\hdots,x_n)$$
for all $a>0$ and $v\in\R^d$ and 
\begin{eqnarray*}
&& \int_{h_n^{-1}(r,y)}  Jh_n(x_1,\hdots,x_n)^{-1} {\cal H}^{nd-d-1}(d(x_1,\hdots,x_n))\\ &=& r^{nd-d-1}\int_{h_n^{-1}(0,1)}  Jh_n(x_1,\hdots,x_n)^{-1} {\cal H}^{nd-d-1}(d(x_1,\hdots,x_n)).
\end{eqnarray*}
Hence, (\ref{eqn:transformation1}) simplifies to
\begin{equation}\label{eqn:transformation2}
||f_n||^2_n\leq C_n\lambda^n \int_0^\infty\int_{\R^d} \overline{f}_n(r,y) (\kappa_dr^d)^{n-1-1/d} dy dr
\end{equation}
for $n\geq 2$ with constants
$$C_n=\kappa_d^{-n+1+1/d}\int_{h_n^{-1}(0,1)}  Jh_n(x_1,\hdots,x_n)^{-1} {\cal H}^{nd-d-1}(d(x_1,\hdots,x_n)).$$

\begin{lemma}\label{lem:boundCn}
The constants $C_n$, $n\geq 2$, are finite and there is a constant $\tilde{c}_d>0$ only depending on the dimension $d$ such that $$C_n \leq \tilde{c}_d {n \choose d+1}(n-1)$$
for $n\geq d+1$.
\end{lemma}
\begin{proof}
A straightforward computation yields
\begin{eqnarray*}
C_n & = & \kappa_d^{-n+1+1/d}\int_{h_n^{-1}(0,1)}  Jh_n(x_1,\hdots,x_n)^{-1} {\cal H}^{nd-d-1}(d(x_1,\hdots,x_n))\\
& = & (n-1)d\kappa_d^{-n+1/d}\int_0^1\int_{B^d(0,1)}\int_{h_n^{-1}(y,r)}\!\!\!\!\!\!\!\!  Jh_n(x_1,\hdots,x_n)^{-1} {\cal H}^{nd-d-1}(d(x_1,\hdots,x_n))dy dr\\
&=& (n-1)d \kappa_d^{-n+1/d}\int_{(\R^d)^n}\1(h_n(x_1,\hdots,x_n)\in [0,1]\times B^d(0,1))dx_1\hdots dx_n<\infty.
\end{eqnarray*}
For almost all $(x_1,\hdots,x_n)\in(\R^d)^n$ at most $d+1$ points are on the boundary of the minimal ball that contains all points, and we assume that these are $x_1,\hdots,x_{d+1}$. Since the center of the minimal ball is in $B^d(0,1)$ and the radius is in $[0,1]$, these point must be in $B^d(0,2)$. The remaining points are in a ball with radius $1$ around a center given by the first $d+1$ points. These considerations lead to the bound
$$C_n \leq (n-1)d \kappa_d^{-n+1/d} {n \choose d+1} (\kappa_d 2^d)^{d+1}\kappa_d^{n-d-1}=\kappa_d^{-1/d}2^{d^2+d}{n \choose d+1}(n-1)$$
for $n\geq d+1$.
\end{proof}
Our main tool for the computation of the right hand side in (\ref{eqn:transformation2}) is the following inequality:
\begin{lemma}\label{lem:intfbar}
There are constants $c_{1,d},c_{2,d}>0$ only depending on the dimension $d$ such that
\begin{equation*}
\int_{\R^d} \overline{f}_n(r,y)^2 dy \leq \exp(-2\lambda\kappa_d r^d) \sum_{i=0}^{d-1}\kappa_{d-i}V_i(K) \left(\frac{c_{1,d}(\kappa_dr^d)^{1-\frac{i}{d}}}{((n-1)!)^2\lambda^2}+\frac{c_{2,d}(\kappa_dr^d)^{-\frac{i}{d}}}{(n!)^2\lambda^3}\right)
\end{equation*}
for all $r>0$ and $n\geq 2$.
\end{lemma}
\begin{proof}
Let $(\partial K)_s=\{y\in \R^d: \rho(\partial K, y)\leq s\}$. Together with (\ref{eqn:Boundfn1}) and (\ref{eqn:Boundfn2}), we obtain 
\begin{eqnarray}\label{eqn:boundtwoparts}
\int_{\R^d} \overline{f}_n(r,y)^2 dy & \leq & \int_{(\partial K)_{8r}}\frac{1}{((n-1)!)^2\lambda^2}\exp(-2\lambda\kappa_d r^d) dy\\ && +\int_{\R^d\setminus (\partial K)_{8r}}\frac{4}{(n!)^2\lambda^2}\exp(-2\lambda\kappa_d \rho(y,\partial K)^d/8^d) dy \notag.
\end{eqnarray}
The volume of $(\partial K)_{8r}\cap K^C$ is given by the Steiner formula (see Theorem 2.2.4 in \cite{SchneiderWeil1992}) and $\Vol((\partial K)_{8r}\cap K)\leq \Vol((\partial K)_{8r}\cap K^C)$ such that 
\begin{eqnarray*}
&&\int_{(\partial K)_{8r}}\frac{1}{((n-1)!)^2\lambda^2}\exp(-2\lambda\kappa_d r^d) dy\\ & \leq & 2\frac{1}{((n-1)!)^2\lambda^2}\exp(-2\lambda\kappa_d r^d) \sum_{i=0}^{d-1}\kappa_{d-i} V_{i}(K)(8r)^{d-i}\notag\\ & \leq & \frac{c_{1,d}}{((n-1)!)^2\lambda^2} \exp(-2\lambda\kappa_d r^d)\sum_{i=0}^{d-1}\kappa_{d-i} V_{i}(K)(\kappa_dr^d)^{1-\frac{i}{d}}. \notag
\end{eqnarray*}
To the second expression in (\ref{eqn:boundtwoparts}) we apply the coarea formula with the Lipschitz-function $u:\R^d\rightarrow\R, x\mapsto \rho(x,\partial K)$. 
\begin{eqnarray*}
&& \int_{\R^d\setminus (\partial K)_{8r}}\frac{4}{(n!)^2\lambda^2}\exp(-2\lambda\kappa_d \rho(y,\partial K)^d/8^d) dy\\ & = & \int_{8r}^\infty \int_{u^{-1}(\delta)} \frac{4}{(n!)^2\lambda^2}\exp(-2\lambda\kappa_d \delta^d/8^d) ||\nabla u(y)||^{-1} {\cal H}^{d-1}(dy) d\delta.
\end{eqnarray*}
It is easy to see that $|\nabla_{v(x)} u(x)|=||v(x)||$, where $\nabla_{v(x)} u(x)$ is the directional derivative in direction $v(x)=x-\proj_K(x)$. Hence, we have $||v(x)||=|\nabla_{v(x)}u(x)|\leq ||\nabla u(x)||\ ||v(x)||$ and $||\nabla u(x)||\geq 1$. By the Steiner formula (see Theorem 2.2.4 in \cite{SchneiderWeil1992}), we know that
\begin{eqnarray*}
{\cal H}^{d-1}(\{z\in\R^d: \rho(z,\partial K)=\delta\}) & \leq & 2{\cal H}^{d-1}(\{z\in K^C: \rho(z,\partial K)=\delta\})\\ &=& \sum_{i=0}^{d-1}(d-i)\kappa_{d-i}V_i(K)\delta^{d-1-i}
\end{eqnarray*}
and altogether we obtain
\begin{eqnarray*}
&&\int_{\R^d\setminus (\partial K)_{8r}}\frac{4}{(n!)^2\lambda^2}\exp(-2\lambda\kappa_d \rho(y,\partial K)^d/8^d) dy\\ & \leq & \int_{8r}^\infty \frac{4}{(n!)^2\lambda^2}\exp(-2\lambda\kappa_d \delta^d/8^d) \sum_{i=0}^{d-1} (d-i)\kappa_{d-i}V_i(K)\delta^{d-1-i} d\delta\\
& \leq & \frac{4}{(n!)^2\lambda^2} \sum_{i=0}^{d-1} (d-i)\kappa_{d-i}V_i(K)(8r)^{-i}\int_{8r}^\infty \exp(-2\lambda\kappa_d \delta^d/8^d)\delta^{d-1}d\delta \\
&=& 2\frac{8^d}{(n!)^2\lambda^3} \exp(-2\lambda\kappa_d r^d) \sum_{i=0}^{d-1}(d-i)\kappa_{d-i}V_i(K)(8r)^{-i}\\ &\leq&\frac{c_{2,d}}{(n!)^2\lambda^3} \exp(-2\lambda\kappa_d r^d) \sum_{i=0}^{d-1}\kappa_{d-i}V_i(K)(\kappa_dr^d)^{-\frac{i}{d}}.
\end{eqnarray*}
\end{proof}
By $h_1(x_1)=(0,x_1)$, Lemma \ref{lem:Boundfn2} and the coarea formula with the same function $u$ as in the previous proof, it follows that
\begin{eqnarray}\label{eqn:upperboundf1}
||f_1||_1^2 & \leq& \lambda\int_{\R^d} \frac{4}{\lambda^2}\exp(-2\lambda\kappa_d\rho(y,\partial K)^d/8^d) dy\\
&=& \frac{4}{\lambda}\int_0^\infty \int_{u^{-1}(r)} \exp(-2\lambda\kappa_d r^d/8^d) ||\nabla u(y)||^{-1}{\cal H}^{d-1}(dy)dr\notag\\
&\leq& \frac{4}{\lambda} \int_0^\infty \exp(-2\lambda\kappa_d r^d/8^d) \sum_{i=0}^{d-1} (d-i)\kappa_{d-i}V_i(K)r^{d-1-i} dr.\notag
\end{eqnarray}
Combining (\ref{eqn:transformation2}) and Lemma \ref{lem:intfbar}, we have
\begin{eqnarray*}
||f_n||_n^2 & \leq & C_n \int_0^\infty \exp(-2\lambda\kappa_d r^d) \sum_{i=0}^{d-1}\kappa_{d-i}V_i(K) \frac{c_{1,d}\lambda^{-2+\frac{i+1}{d}}}{((n-1)!)^2}(\lambda\kappa_dr^d)^{n-\frac{i+1}{d}} dr\\
&&+ C_n \int_0^\infty \exp(-2\lambda\kappa_d r^d) \sum_{i=0}^{d-1}\kappa_{d-i}V_i(K)  \frac{c_{2,d}\lambda^{-2+\frac{i+1}{d}}}{n!}(\lambda\kappa_dr^d)^{n-1-\frac{i+1}{d}}dr.
\end{eqnarray*}
for $n\geq 2$. Comparing this with (\ref{eqn:upperboundf1}), we see that the same bound also holds for $n=1$ if the constant $C_1$ is chosen appropriately. Now substitution and the definition of the Gamma function lead to
\begin{eqnarray*}
||f_n||_n^2 & \leq & C_n \frac{c_{1,d}}{((n-1)!)^2}\sum_{i=0}^{d-1}\kappa_{d-i}V_i(K)\frac{\lambda^{-2+\frac{i+1}{d}}}{2^{n-\frac{i+1}{d}}}\frac{1}{(2\lambda\kappa_d)^{\frac{1}{d}}}\int_0^\infty \exp(-y) y^{n-1-\frac{i}{d}}dy\\
&&+C_n \frac{c_{2,d}}{(n!)^2}\sum_{i=0}^{d-1}\kappa_{d-i}V_i(K)\frac{\lambda^{-2+\frac{i+1}{d}}}{2^{n-1-\frac{i+1}{d}}}\frac{1}{(2\lambda\kappa_d)^{\frac{1}{d}}}\int_0^\infty \exp(-y) y^{n-2-\frac{i}{d}}dy\\
&=& C_n \sum_{i=0}^{d-1} \kappa_{d-i}V_i(K)\left(\frac{\tilde{c}_{1,d}\Gamma(n-\frac{i}{d})}{((n-1)!)^2 2^n}+\frac{c_{2,d}\Gamma(n-1-\frac{i}{d})}{(n!)^2 2^n}\right)\lambda^{-2+\frac{i}{d}}
\end{eqnarray*}
with constants $\tilde{c}_{1,d},\tilde{c}_{2,d}>0$ and it is easy to see that
\begin{equation}\label{eqn:boundnormfn}
n! ||f_n||_n^2 \leq \frac{1}{2^{n}}C_n \left(\tilde{c}_{1,d}n+\tilde{c}_{2,d}\frac{1}{n(n-1)}\right)\sum_{i=0}^{d-1} \kappa_{d-i}V_i(K) \lambda^{-2+\frac{i}{d}}.
\end{equation}
By Lemma \ref{lem:boundCn}, we know that $C_n$ is bounded by a polynomial of order $d+2$ in $n$ and it follows directly that the series
$$\sum_{n=1}^\infty\frac{1}{2^{n}}C_n \left(\tilde{c}_{1,d}n+\tilde{c}_{2,d}\frac{1}{n(n-1)}\right)$$
converges, which proves the upper bound in Theorem \ref{thm:variance} and that the condition (\ref{eqn:conditionkernels}) in Theorem \ref{thm:abstractCLT} is satisfied for the Poisson-Voronoi approximation.

In order to conclude the proof of Theorem \ref{thm:variance}, it remains to construct a lower bound. Because of (\ref{eqn:chaosvariance}) it is sufficient to give a lower bound for $||f_1||_1^2$.
\begin{lemma}\label{lem:lowerf1}
There is a constant $\underline{C}$ only depending on the dimension $d$ such that
\begin{equation}\label{eqn:LowerBoundVariance}
||f_1||_1^2 \geq \underline{C} \kappa_1 V_{d-1}(K)\lambda^{-1-\frac{1}{d}}
\end{equation}
for $\lambda\geq (2/r_K)^d$, where $r_K$ is the inradius of $K$.
\end{lemma}
\begin{proof}
Recall that $B^d(y,\delta)\subset\R^d$ stands for a ball with center $y$ and radius $\delta>0$.
We consider the set $$M_{\varepsilon}=\left\{x\in K^{C}: \rho(x,K)\leq \varepsilon, \Vol((B^d(x,2\varepsilon)\setminus B^d(x,\varepsilon))\cap K)\geq \frac{\kappa_d \varepsilon^d}{2^d} \right\}$$
for $\varepsilon\leq r_K/2$. By Lemma 4 in \cite{SchuettWerner1990}, it is known that
$${\cal H}^{d-1}\left(\left\{x\in\partial K: r(x)\geq \varepsilon\right\}\right)\geq \left(1-\frac{\varepsilon}{r_K}\right)^{d-1}\kappa_1V_{d-1}(K),$$
where $r(x)$ is the radius of the largest ball that is contained in $K$ and contains $x$. It is easy to see that $x\in K^C$ with $\rho(x,K)\leq \varepsilon$ is in $M_{\varepsilon}$ if $r(\proj_K(x))\geq \varepsilon$. As a consequence, we have
\begin{eqnarray}\label{eqn:lowerBoundMepsilon}
\Vol(M_{\varepsilon}) &\geq & {\cal H}^{d-1}\left(\left\{x\in\partial K: r(x)\geq\varepsilon\right\}\right)\varepsilon\\ &\geq & \left(1-\frac{\varepsilon}{r_K}\right)^{d-1}\kappa_1V_{d-1}(K)\varepsilon\geq \frac{1}{2^d}\kappa_1 V_{d-1}(K)\varepsilon.\notag
\end{eqnarray}
For $x\in M_{\varepsilon}$ it holds
\begin{equation}\label{eqn:lowerboundf1}
|f_1(x)|\geq \frac{\kappa_d\varepsilon^d}{2^d} \exp(-\left(4^d -2^{-d}\right) \lambda\kappa_d\varepsilon^d)\left(1-\exp(-2^{-d}\ \lambda\kappa_d\varepsilon^d)\right).
\end{equation}
To see (\ref{eqn:lowerboundf1}), the underlying idea is that for every $x\in M_{\varepsilon}$ there is, by definition of $M_{\varepsilon}$, a set $U\subset (B^d(x,2\varepsilon)\setminus B^d(x,\varepsilon))\cap K$ with $\ell_d(U)=2^{-d}\kappa_d\varepsilon^d$. Then
$$\P(\eta(B^d(x,4\varepsilon)\setminus U)=0, \eta(U)\geq 1)=\exp(-(4^d-2^{-d})\lambda\kappa_d\varepsilon^d)\left(1-\exp(-2^{-d}\ \lambda\kappa_d\varepsilon^d)\right)$$
and for this event the effect of adding $x$ to the point process is larger than $\frac{\kappa_d\varepsilon^d}{2^d}$.
Combining (\ref{eqn:lowerBoundMepsilon}) and (\ref{eqn:lowerboundf1}), we obtain
\begin{eqnarray*}
||f_1||_1^2 &=&\lambda\int_{\R^d} f_1(x)^2 dx\geq \lambda\int_{M_{\varepsilon}} f_1(x)^2 dx\\
&\geq& \lambda \frac{1}{2^d}  \kappa_1V_{d-1}(K)\varepsilon\ \frac{\kappa_d^2\varepsilon^{2d}}{4^d} \exp(-2(4^d-2^{-d})\lambda\kappa_d\varepsilon^d)\left(1-\exp(2^{-d}\ \lambda\kappa_d\varepsilon^d)\right)^2
\end{eqnarray*}
and the choice $\varepsilon=\lambda^{-\frac{1}{d}}$ leads to
$$||f_1||_1^2 \geq \frac{\kappa_d^2}{8^d} \exp(-2(4^d-2^{-d})\kappa_d)\left(1-\exp(-2^{-d}\ \kappa_d)\right)^2 \kappa_1 V_{d-1}(K)\lambda^{-1-\frac{1}{d}}.$$
\end{proof}
\begin{remark}\rm
An inequality as (\ref{eqn:LowerBoundVariance}) cannot hold for all $\lambda>0$ as the following consideration shows:\\
We fix a compact convex set $K_0$ with $0\in K$ and a compact window $W\supset K$ and set $K_r=rK=\{r x: x\in K_0 \}$ for $r>0$. We define the random variable $\widetilde{\PV}(W)$ as
$$\widetilde{\PV}(W)=\Vol\left(\left\{y\in \R^d: \rho(y,W)\leq||y-x||\ \forall x\in \eta\cap W^C\right\}\right).$$
A short computation proves $\E\ \widetilde{\PV}(W)^2<\infty$. Then, it holds
\begin{eqnarray*}
\V \PV(K_r) &=& \E\PV(K_r)^2-\Vol(K_r)^2\leq \E\PV(K_r)^2\\ &\leq& \left(1-\exp(-\lambda \Vol(K_r))\right)\E\ \widetilde{\PV}(W)^2.
\end{eqnarray*}
For $r\rightarrow 0$, the right hand side has order $r^d$, whereas $V_{d-1}(K_r)$ is only of order $r^{d-1}$.
\end{remark}

\section{Proof of Theorem \ref{thm:CLT}} \label{sec:CLTPV}

In this section, we use our abstract central limit theorem Theorem \ref{thm:abstractCLT} to prove Theorem \ref{thm:CLT}. Since it follows by (\ref{eqn:boundnormfn}) that the condition (\ref{eqn:conditionkernels}) is satisfied, it remains only to check (\ref{eqn:conditionRs}), which requires
$$\frac{\sqrt{R_{ij}}}{\V \PV(K)},\frac{\sqrt{\tilde{R}_i}}{\V \PV(K)}\rightarrow 0 \quad \text{ as } \lambda\rightarrow \infty.$$ 
We show that for every $\pi\in \overline{\Pi}_{i,j}$
$$M_{\lambda}=\lambda^{|\pi|}\int_{\R^d}\hdots\int_{\R^d}|(f_i*f_i*f_j*f_j)_{\pi}(y_1,\hdots,y_{|\pi|})|dy_1\hdots dy_{|\pi|}$$
converges to zero as $\lambda\rightarrow\infty$ at a sufficiently high rate such that the inequalities (\ref{eqn:boundRij}) and (\ref{eqn:boundRi}) in Proposition \ref{prop:product} imply that condition (\ref{eqn:conditionRs}) is satisfied.

We define functions $g_n: \left(\R^d\right)^n\rightarrow\R, n\in\N$, as
$$g_n(x_1,\hdots,x_n)=\max\left\{\diam(x_1,\hdots,x_n),\max_{i=1,\hdots,n}\rho(x_i,\partial K)\right\},$$ 
where $\diam(x_1,\hdots,x_n)$ stands for the diameter of $x_1,\hdots,x_n$. Using this notation, we can state the following upper bound for $f_n$:

\begin{lemma}\label{lem:boundfng}
Let $x_1,\hdots,x_n\in\R^d$ and $\delta=g_n(x_1,\hdots,x_n)$. Then
\begin{equation*}
|f_n(x_1,\hdots,x_n)|\leq \frac{2}{n!\lambda} \exp(-\lambda\kappa_d\delta^d/4^d)=:\tilde{f}_n(\delta).
\end{equation*}
\end{lemma}
\begin{proof}
Without loss of generality we can assume $\rho(x_1,\partial K)=\delta$ or $\rho(x_1,x_2)=\delta$. For the first case, let $\tilde{x}=\frac{1}{2}(x_1+\proj_{\partial K}(x_1))$, where $\proj_{\partial K}(x_1)$ is the projection of $x_1$ on the boundary of $K$. If the projection is not unique (this can happen for $x_1\in K$), it does not matter which of the points is taken. Then, it holds
$$\frac{\delta}{4} \leq \rho(y,\tilde{x})\leq\frac{3}{4}\delta \leq \rho(y,\partial K)\ \text{ for all }y\in B^d(x_1,\delta/4).$$
Hence, it follows from Lemma \ref{lem:kernels} and a straight forward computation as in the proof of Lemma \ref{lem:Boundfn1} that
\begin{eqnarray*}
|f_n(x_1,\hdots,x_n)| & \leq & \frac{1}{n!} \int_{\R^d\setminus B^d(x_1,\delta/4)} \P(z(y,\eta)\notin B^d(y,||y-x_1||))dy\\ &&+\frac{1}{n!}\int_{B^d(x_1,\delta/4)}\P(z(y,\eta)\notin B^d(y,\rho(y,\partial K)||))dy\\
& \leq & \frac{1}{n!} \int_{\R^d\setminus B^d(x_1,\delta/4)} \P(z(y,\eta)\notin B^d(y,||y-x_1||))dy\\ &&+\frac{1}{n!}\int_{\R^d\setminus B^d(\tilde{x},\delta/4)}\P(z(y,\eta)\notin B^d(y,||y-\tilde{x}||))dy\\
&=& \frac{2}{n!\lambda} \exp(-\lambda\kappa_d\delta^d/4^d).
\end{eqnarray*}
In the case $\rho(x_1,x_2)=\delta$, we replace $\tilde{x}$ by $x_2$ and obtain the same bound.
\end{proof}\\
We prepare the application of the coarea formula by showing the following properties of $g_n:$

\begin{lemma}\label{lem:gn}
\begin{itemize}
 \item [a)] $g_n$ is a Lipschitz function with $||\nabla g_n||\geq 1$ almost everywhere.
\item [b)] There is a constant $c_d>0$ only depending on the dimension $d$ such that
\begin{eqnarray}\label{eqn:Hgn}
{\cal H}^{nd-1}(g_n^{-1}(\delta)) & \leq & \sum_{i=0}^{d-1}\kappa_{d-i}V_i(K)\delta^{d-i}d\kappa_d\delta^{d-1}(\kappa_d\delta^d)^{n-2}\\ &&+n\sum_{i=0}^{d-1}(d-i)\kappa_{d-i}V_i(K)\delta^{d-1-i}(\kappa_d\delta^d)^{n-1}\notag\\ &\leq& c_d n\sum_{i=0}^{d-1} \kappa_{d-i}V_i(K)(\kappa_d\delta^d)^{n-\frac{i+1}{d}}.\notag
 \end{eqnarray}
\end{itemize}
\end{lemma}

\begin{proof}
$g_n(x_1,\hdots,x_n)$ is always given by the distance of two points or by the distance of a point to the boundary of $K$. If we move one of these points exactly in the opposite direction $v$ of the second point or the boundary of $K$, the directional derivative is $\nabla_vg_n(x_1,\hdots,x_n)=||v||$ and
$$|\nabla_vg_n(x_1,\hdots,x_n)|\leq ||\nabla g_n(x_1,\hdots,x_n)||\ ||v||$$
implies $||\nabla g_n(x_1,\hdots,x_n)||\geq 1$ and thus a).

For the proof of b) we consider the same situations as in the proof of a). If there are two points $x_i,x_j\in\R^d$ such that $\rho(x_i,x_j)=\delta$, $x_i$ must be in $(\partial K)_{\delta}$, $x_j$ in a sphere around $x_i$ with radius $\delta$ and the remaining $n-2$ points must be in a ball with radius $\delta$ and center $x_i$. If $\rho(x_i,\partial K)=\delta$, $x_i$ must be in the set $\{y\in\R^d: \rho(y,\partial K)=\delta\}$ and the remaining points are in a ball with radius $\delta$ and center $x_i$. Combining these considerations with the Steiner formula (see Theorem 2.2.4 in \cite{SchneiderWeil1992}) yields (\ref{eqn:Hgn}).
\end{proof}\\
For $l=1,2,3,4$ let $\pi_l(y_1,\hdots,y_{|\pi|})\subset\{y_1,\hdots,y_{|\pi|}\}$ be the new variables that occur in the $l$-th function of $(f_i*f_i*f_j*f_j)_{\pi}$ and let $|\pi_l|$ stand for the number of these variables. We set $r=g_{|\pi|}(y_1,\hdots,y_{|\pi|})$ and $$\delta_1=g_{|\pi_1|}(\pi_1(y_1,\hdots,y_{|\pi|})),\hdots,\delta_4=g_{|\pi_4|}(\pi_4(y_1,\hdots,y_{|\pi|})).$$
Since $\pi_l(y_1,\hdots,y_{|\pi|})\subset \{y_1,\hdots,y_{|\pi|}\}$, it is easy to see that $\delta_l=g_{|\pi_l|}(\pi_l(y_1,\hdots,y_{|\pi|}))\leq g_{|\pi|}(y_1,\hdots,y_{|\pi|})=r$ for $l=1,2,3,4$. If there is a $y_j$ with $\rho(y_j,\partial K)=r$, we have at least two $l_1,l_2\in\{1,2,3,4\}$ such that $y_j\in \pi_{l_1}(y_1,\hdots,y_{|\pi|})$ and $y_j\in \pi_{l_2}(y_1,\hdots,y_{|\pi|})$, which implies $\delta_{l_1}=\delta_{l_2}=r$. The other case is that there are $y_{j_1}$ and $y_{j_2}$ such that $\rho(y_{j_1},y_{j_2})=r$. If there is a $l\in\{1,2,3,4\}$ with $y_1,y_2\in \pi_l(y_1,\hdots,y_{|\pi|})$, it follows directly $\delta_l=r$. Otherwise, $\pi\in\overline{\Pi}_{i,j}$ implies that we have a $y_{j_3}$ and $l_1,l_2\in \{1,2,3,4\}$ with $y_{j_1},y_{j_3}\in \pi_{l_1}(y_1,\hdots,y_{|\pi|})$ and $y_{j_2},y_{j_3}\in \pi_{l_2}(y_1,\hdots,y_{|\pi|})$. By the inequality $r=\rho(y_{j_1},y_{j_2})\leq \rho(y_{j_1},y_{j_3})+\rho(y_{j_3},y_{j_2})$, it follows $\max\{\delta_{l_1},\delta_{l_2}\}\geq \max\{\rho(y_{j_1},y_{j_3}),\rho(y_{j_3},y_{j_2})\}\geq r/2$. Hence, it holds
$r/2\leq\max_{l=1,\hdots,4}\delta_l\leq r$.

Together with the coarea formula Theorem \ref{thm:coarea}, Lemma \ref{lem:boundfng} and Lemma \ref{lem:gn}, we obtain
\begin{eqnarray*}
M_{\lambda} & \leq & \lambda^{|\pi|}\int_{(\R^d)^{|\pi|}} \tilde{f}_i(\delta_1)\tilde{f}_i(\delta_2)\tilde{f}_j(\delta_3)\tilde{f}_j(\delta_4) dy_1\hdots dy_{|\pi|}\\
& \leq & \lambda^{|\pi|}\int_{(\R^d)^{|\pi|}} \tilde{f}_i(\delta_1)\tilde{f}_i(\delta_2)\tilde{f}_j(\delta_3)\tilde{f}_j(\delta_4) ||\nabla g_{|\pi|}|| dy_1\hdots dy_{|\pi|}\\ 
&=& \lambda^{|\pi|} \int_0^\infty\int_{g_n^{-1}(r)}\tilde{f}_i(\delta_1)\tilde{f}_i(\delta_2)\tilde{f}_j(\delta_3)\tilde{f}_j(\delta_4) {\cal H}^{|\pi|d-1}(d(y_1,\hdots,y_{|\pi|}))dr\\
&\leq& \lambda^{|\pi|-4} \frac{16}{(i!)^2(j!)^2}\int_0^\infty\exp(-\lambda\kappa_d r^d/8^d)c_d |\pi|\sum_{i=0}^{d-1} \kappa_{d-i}V_i(K)(\kappa_d r^d)^{|\pi|-\frac{i+1}{d}}dr.
\end{eqnarray*}
By substitution and the definition of the Gamma function, we have
\begin{eqnarray*}
M_\lambda &\leq & \frac{16c_d|\pi|}{(i!)^2(j!)^2} \sum_{i=0}^{d-1}  \kappa_{d-i}V_i(K)8^{|\pi|d-i-1}\lambda^{\frac{i+1}{d}-4}\int_0^\infty \exp(-\lambda\kappa_d r^d/8^d)(\lambda\kappa_d r^d/8^d)^{|\pi|-\frac{i+1}{d}}dr\\
&=& \frac{16c_d|\pi|}{d(i!)^2(j!)^2} \sum_{i=0}^{d-1}  \kappa_{d-i}V_i(K)\lambda^{\frac{i}{d}-4}8^{|\pi|d-i}\kappa_d^{-\frac{1}{d}}\int_0^\infty \exp(-y)y^{|\pi|-1-\frac{i}{d}}dy\\
&=& \frac{16c_d|\pi|}{d(i!)^2(j!)^2} \sum_{i=0}^{d-1}  \kappa_{d-i}V_i(K)\lambda^{\frac{i}{d}-4}8^{|\pi|d-i}\kappa_d^{-\frac{1}{d}}\Gamma(|\pi|-i/d).
\end{eqnarray*}
Thus, each $M_{\lambda}$ has the order $\lambda^{-3-\frac{1}{d}}$ and, by Proposition \ref{prop:product}, all $R_{ij}$ and $\tilde{R}_i$ have at most the same order. As a consequence of Theorem \ref{thm:variance}, $\V \PV(K)$ has the order $\lambda^{-1-\frac{1}{d}}$, which means that $\frac{\sqrt{R_{ij}}}{\V \PV(K)}$ and $\frac{\sqrt{\tilde{R}_i}}{\V \PV(K)}$ have a order less or equal than $\lambda^{-\frac{1}{2}+\frac{1}{2d}}$ and 
$$\frac{\sqrt{R_{ij}}}{\V \PV(K)}, \frac{\sqrt{\tilde{R}_i}}{\V \PV(K)}\rightarrow 0 \quad \text{ as }\lambda\rightarrow\infty.$$
Now all assumptions of Theorem \ref{thm:abstractCLT} are satisfied for the volume of the Poisson-Voronoi approximation and Theorem \ref{thm:CLT} is a direct consequence.

\begin{remark}\label{remarkConvexity}\rm
In Theorem \ref{thm:CLT} and Theorem \ref{thm:variance}, we assume that the approximated set $K$ is convex. But the convexity is not necessary for the construction of the Poisson-Voronoi approximation such that it is a natural question, if one can extend our results to more general set classes. 
The convexity assumption is only needed to bound the volume and the surface area of the parallel sets $(\partial K)_r$ by the Steiner formula in the proofs of Theorem \ref{thm:variance} and Theorem \ref{thm:CLT} and to apply Lemma 4 from \cite{SchuettWerner1990} in the proof of Lemma \ref{lem:lowerf1}. Hence, one can extend the results to compact sets $M\subset\R^d$ that satisfy the following additional assumptions:
\begin{description}
 \item [(S1)] There are constants $c^{(i)}_M,i=1,\hdots,d,$ depending on $M$ such that
$$\Vol((\partial M)_r)\leq\sum_{i=1}^{d} c^{(i)}_M r^i\ \text{ and }\ {\cal H}^{d-1}(\partial((\partial M)_r))\leq \sum_{i=1}^{d} c^{(i)}_M r^{i-1}$$
with $(\partial M)_r=\{x\in\R^d: \rho(x,\partial M)\leq r\}$ for $r>0$.
\item [(S2)] It holds
$$\liminf\limits_{r\rightarrow 0}\Vol(\tilde M_r)/r>0$$
with $\tilde M_r=\{x\in M^C: \rho(x,M)\leq r, \Vol((B^d(x,2r)\setminus B^d(x,r))\cap M)\geq\frac{\kappa_dr^d}{2^d}\}$.
\end{description}
Assumption (S1) allows us to bound the volume and the surface area of the parallel sets $(\partial M)_r$ by a kind of Steiner formula. In the upper bound in Theorem \ref{thm:variance}, the intrinsic volumes must be replaced by the constants $c^{(i)}_M,i=1,\hdots,d$. Our proof of the lower bound in Theorem \ref{thm:variance} requires assumption (S2), which replaces a rolling ball result for convex sets from \cite{SchuettWerner1990}. Then the constant $\underline{C}$ in (\ref{eqn:variance}) and the lower bound for $\lambda$ depend on the limit inferior in (S2).

Since (S1) and (S2) are obviously true for convex sets, they still hold for polyconvex sets and, of course, for all polytopes.
\end{remark}

\paragraph{Acknowledgement} The author would like to thank Matthias Reitzner and Christoph Th\"ale for some useful discussions and valuable remarks on the topic of this paper.

\end{document}